\documentclass{amsart} 
\usepackage{amssymb,latexsym,textcomp}

\usepackage[OT2,T1]{fontenc}

\theoremstyle{plain}
\newtheorem{theorem}{Theorem}[section]
\newtheorem{lemma}[theorem]{Lemma}
\newtheorem{corollary}[theorem]{Corollary}
\newtheorem{proposition}[theorem]{Proposition}

\theoremstyle{definition}
\newtheorem{definition}{Definition}[section]
\newtheorem{example}{Example}[section]

\theoremstyle{remark}
\newtheorem{remark}{Remark}[section]

\DeclareMathOperator{\tr}{tr}%
\DeclareMathOperator{\re}{Re}
\newcommand{\N}{\mathbf N}

\newcommand{\R}{\mathbf R}

\newcommand{\f}{\varphi}
\newcommand{\e}{\varepsilon}
\renewcommand{\d}{\mathrm d}
\newcommand{\skp}[2]{\left<#1,#2\right>}
\newcommand{\C}{\mathbf C}
\renewcommand{\a}{\alpha}
\newcommand{\luin}{\left|\!\left|\!\left|}
\newcommand{\ruin}{\right|\!\right|\!\right|}
\newcommand{\ovl}{\overline}
\newcommand{\Sc}{\mathcal C}

\numberwithin{equation}{section} 

\begin{document}
\title[The applications of Cauchy-Schwartz inequality]{The applications of Cauchy-Schwartz inequality for Hilbert modules to elementary operators and i.p.t.i.\ transformers} 

\author{Dragoljub J. Ke\v cki\' c}
\address{University of Belgrade\\ Faculty of Mathematics\\ Student\/ski trg 16-18\\ 11000 Beograd\\ Serbia}

\email{keckic@matf.bg.ac.rs}

\thanks{The author was supported in part by the Ministry of education and science, Republic of Serbia, Grant \#174034.}

\begin{abstract}
We apply the inequality $\left|\left<x,y\right>\right|\le||x||\,\left<y,y\right>^{1/2}$ to give an easy and elementary proof of many operator inequalities for elementary operators and inner type product integral transformers obtained during last two decades, which also generalizes all of them.
\end{abstract}




\maketitle

\section{Introduction}

Let $A$ be a Banach algebra, and let $a_j$, $b_j\in A$. Elementary operators, introduced by Lummer and Rosenblum in \cite{LumerRosenblum} are mappings from $A$ to $A$ of the form
\begin{equation}\label{ElOp}
x\mapsto\sum_{j=1}^na_jxb_j.
\end{equation}
Finite sum may be replaced by infinite sum provided some convergence condition.

A similar mapping, called inner product type integral transformer (i.p.t.i.\ transformers in further), considered in \cite{BananaJFA2005}, is defined by
\begin{equation}\label{InnerTT}
X\mapsto\int_\Omega\mathcal A_tX\mathcal B_t\d\mu(t),
\end{equation}
where $(\Omega,\mu)$ is a measure space, and $t\mapsto\mathcal A_t$, $\mathcal B_t$ are fields of operators in $B(H)$.

During last two decades, there were obtained a number of inequalities involving elementary operators on $B(H)$ as well as i.p.t.i.\ type transformers. The aim of this paper is two give an easy and elementary proof of those proved in \cite{BananaPAMS1998,BananaJLMS1999,BananaJFA2005,Ludaci5autora,BananaGeorg,BananaStefan,LudaciFilomat} and \cite{StefanAOT2016} using the Cauchy Schwartz inequality for Hilbert $C^*$-modules -- the inequality stated in the abstract, which also generalizes all of them.

\section{Preliminaries}

Throughout this paper $A$ will always denote a semifinite von Neumann algebra, and $\tau$ will denote a semifinite trace on $A$. $L^p(A;\tau)$ will denote the non-commutative $L^p$ space, $L^p(A;\tau)=\{a\in A~|~\|a\|_p=\tau(|a|^p)^{1/p}<+\infty\}$.

It is well known that $L^1(A;\tau)^*\cong A$, $L^p(A;\tau)^*\cong L^q(A;\tau)$, $1/p+1/q=1$. Both dualities are realized by
$$L^p(A;\tau)\ni a\mapsto\tau(ab)\in\C,\qquad b\in L^q(A;\tau)~\mbox{or}~b\in A.$$

For more details on von Neumann algebras the reader is referred to \cite{Takesaki}, and for details on $L^p(A,\tau)$ to \cite{Kosaki}.

Let $M$ be a \emph{right} Hilbert $W^*$-module over $A$. (Since $M$ is right we assume that $A$-valued inner product is $A$-linear in second variable, and adjoint $A$-linear in the first.) We assume, also, that there is a faithful left action of $A$ on $M$, that is, an embedding (and hence an isometry) of $A$ into $B^a(M)$ the algebra of all adjointable bounded $A$-linear operators on $M$. Hence, for $x$, $y\in M$ and $a$, $b\in A$ we have
$$\skp xya=\skp x{ya},\quad\skp{xa}y=a^*\skp xy,\quad\skp x{ay}=\skp{a^*x}y.$$

For more details on Hilbert modules, the reader is referred to \cite{Lance} or \cite{Manuilov}.

We quote the basic property of $A$-valued inner product, a variant od Cauchy-Schwartz inequality.

\begin{proposition} Let $M$ be a Hilbert $C^*$-module over $A$. For any $x$, $y\in M$ we have
\begin{equation}\label{CS}
|\skp xy|^2\le\|x\|^2\skp yy,\qquad |\skp xy|\le\|x\|\skp yy^{1/2},
\end{equation}
in the ordering of $A$.
\end{proposition}

The proof can be found in \cite[page 3]{Lance} or \cite[page 3]{Manuilov}. Notice: $1^\circ$~the left inequality implies the right one, since $t\mapsto t^{1/2}$ is operator increasing function; $2^\circ$~Both inequalities holds for $A$-valued \emph{semi-inner} product, i.e.\ even if $\skp\cdot\cdot$ may be degenerate.

Finally, we need a counterpart of Tomita modular conjugation.

\begin{definition} Let $M$ be a Hilbert $W^*$-module over a semifinite von Neumann algebra $A$, and let there is a left action of $A$ on $M$.

A (possibly unbounded) mapping $J$, defined on some $M_0\subseteq M$ with values in $M$, we call modular conjugation if it satisfies: ($i$) $J(axb)=b^*J(x)a^*$; ($ii$) $\tau(\skp{J(y)}{J(x)})=\tau(\skp xy)$ whenever $\skp xx$, $\skp yy$, $\skp{J(x)}{J(x)}$, $\skp{J(y)}{J(y)}\in L^1(A,\tau)$.

In what follows, we shall use simpler notation $\ovl x$ instead of $J(x)$. Thus, the determining equalities become
\begin{equation}\label{Conjug}
\ovl{axb}=b^*\ovl xa^*,\qquad\tau(\skp{\ovl y}{\ovl x})=\tau(\skp xy).
\end{equation}

The module $M$ together with left action of $A$ and the modular conjugation $J$ we shall call \emph{conjugated $W^*$-module}.
\end{definition}

\begin{definition} Let $M$ be a conjugated $W^*$-module over $A$. We say that $x\in M_0$ is \emph{normal}, if ($i$) $\skp xx x=x\skp xx$, ($ii$) $\skp xx=\skp{\ovl x}{\ovl x}$.
\end{definition}

\begin{remark}It might be a nontrivial question, whether $J$ can be defined on an arbitrary Hilbert $W^*$-module in a way similar to the construction of Tomita's modular conjugation (see \cite{Tomita}). However for our purpose, the preceding definition is enough.
\end{remark}

Examples of conjugated modules are following.

\begin{example} Let $A$ be a semifinite von Neumann algebra, and let $M=A^n$. For $x=(x_1,\dots,x_n)$, $y=(y_1,\dots,y_n)\in M$, $a\in A$, define right multiplication, left action of $A$, the $A$-valued inner product and modular conjugation by
\begin{equation}\label{LeftRight}
xa=(x_1a,\dots,x_na),\qquad ax=(ax_1,\dots,ax_n);
\end{equation}
\begin{equation}\label{InnerConjugate}
\skp xy=x_1^*y_1+\dots+x_n^*y_n,\qquad \ovl x=(x_1^*,\dots,x_n^*).
\end{equation}
All required properties are easily verified. The element $x=(x_1,\dots,x_n)$ is normal whenever all $x_j$ are normal and mutually commute.

We have
$$\skp x{ay}=\sum_{j=1}^nx_j^*ay_j,$$
which is the term of the form (\ref{ElOp}).
\end{example}

There are two important modules with infinite number of summands.

\begin{example}\label{Standard} Let $A$ be a semifinite von Neumann algebra. We consider the standard Hilbert module $l^2(A)$ over $A$ and its dual module $l^2(A)'$ defined by
$$l^2(A)=\Big\{(x_1,\dots,x_n,\dots)~\Big|~\sum_{k=1}^{+\infty}a_k^*a_k~\mbox{converges in norm of}~A\Big\}.$$
$$l^2(A)'=\Big\{(x_1,\dots,x_n,\dots)~\Big|~\Big\|\sum_{k=1}^na_k^*a_k\Big\|\le M<+\infty\Big\}.$$
(It is clear that $x\in l^2(A)'$ if and only if the series $\sum x_k^*x_k$ weakly converges.)

The basic operation on these modules are given by (\ref{LeftRight}) and (\ref{InnerConjugate}) with infinite number of entries.

The main difference between $l^2(A)$ ($l^2(A)'$ respectively) and $A^n$ is the fact that $\ovl x=(x_1^*,\dots,x_n^*,\dots)$ is defined only on the subset of $l^2(A)$ consisting of those $x\in M$ for which $\sum x_kx_k^*$ converges in the norm of $A$.

The element $x=(x_1,\dots,x_n,\dots)\in M_0$ is normal whenever all $x_j$ are normal and mutually commute.
\end{example}

\begin{remark} The notation $l^2(A)'$ comes from the fact that $l^2(A)'$ is isomorphic to the module of all adjointable bounded $A$-linear functionals $\Lambda:M\to A$.

For more details on $l^2(A)$ or $l^2(A)'$ see \cite[\S1.4 and \S2.5]{Manuilov}.
\end{remark}

\begin{example} Let $A$ be a semifinite von Neumann algebra and let $(\Omega,\mu)$ be a measure space.
Consider the space $L^2(\Omega,A)$ consisting of all weakly-$*$ measurable functions such that $\int_\Omega\|x\|^2\d\mu<+\infty$. The weak-$*$ measurability is reduced to the measurability of functions $\f(x(t))$ for all normal states $\f$, since the latter generate the predual of $A$.

Basic operations are given by
$$x(t)\cdot a=x(t)a,\quad a\cdot x(t)=ax(t),\quad\skp xy=\int_\Omega x(t)^*y(t)\d\mu(t),\quad\ovl x(t)=x(t)^*.$$

All required properties are easily verified. Mapping $x\mapsto\ovl x$ is again defined on a proper subset of $L^2(\Omega,A)$. The element $x$ is normal if $x(t)$ is normal for almost all $t$, and $x(t)x(s)=x(s)x(t)$ for almost all $(s,t)$.

Again, for $a\in A$ we have
$$\skp x{ay}=\int_\Omega x(t)^*ay(t)\d\mu(t),$$
which is the term of the form (\ref{InnerTT}).
\end{example}

Thus, norm estimates of elementary operators (\ref{ElOp}), or i.t.p.i.\ transformers (\ref{InnerTT}) are estimates of the term $\skp x{ay}$.

In section 4 we need two more examples.

\begin{example}\label{Interior} Let $M_1$ and $M_2$ be conjugated $W^*$-modules over a semifinite von Neumann algebra $A$. Consider the \emph{interior} product of Hilbert modules $M_1$ and $M_2$ constructed as follows.
The linear span of $x_1\otimes x_2$, $x_1\in M_1$, $x_2\in M_2$ subject to the relations
$$a(x_1\otimes x_2)=ax_1\otimes x_2,\qquad x_1a\otimes x_2=x_1\otimes ax_2,\qquad (x_1\otimes x_2)a=x_1\otimes x_2a,$$
and usual bi-linearity of $x_1\otimes x_2$, can be equipped by an $A$-valued semi-inner product
\begin{equation}\label{InteriorInner}
\skp{x_1\otimes x_2}{y_1\otimes y_2}=\skp{x_2}{\skp{x_1}{y_1}y_2}.
\end{equation}

The completion of the quotient of this linear span by the kernel of (\ref{InteriorInner}) is denoted by $M_1\otimes M_2$ and called \emph{interior tensor product} of $M_1$ and $M_2$. For more details on tensor products, see \cite[Chapter 4]{Lance}.

If $M_1=M_2=M$, $M\otimes M$ can be endowed with a modular conjugation by
$$\ovl{x_1\otimes x_2}=\ovl{x_2}\otimes\ovl{x_1}.$$
All properties are easily verified. Also, $x$ normal implies $x\otimes x$ is normal and $\skp{x\otimes x}{x\otimes x}=\skp xx^2$.
\end{example}

\begin{example}\label{Fock}Let $M_n$, $n\in\N$ be conjugated modules. Their infinite direct sum $\bigoplus_{n=1}^{+\infty}M_n$ is the module consisting of those sequences $(x_n)$, $x_n\in M_n$ such that $\sum_{n=1}^{+\infty}\skp{x_n}{x_n}$ weakly converges, with the $A$-valued inner product
$$\skp{(x_n)}{(y_n)}=\sum_{n=1}^{+\infty}\skp{x_n}{y_n}.$$

The modular conjugation can be given by $\ovl{(x_n)}=(\ovl{x_n})$. Specially, we need the \emph{full Fock module}
$$F=\bigoplus_{n=0}^{+\infty}M^{\otimes n},$$
where $M^{\otimes0}=A$, $M^{\otimes1}=M$, $M^{\otimes2}=M\otimes M$, $M^{\otimes3}=M\otimes M\otimes M$, etc.

For $x\in M$, $\|x\|<1$ the element $\sum_{n=0}^{+\infty}x^{\otimes n}\in F$ (where $x^{\otimes0}:=1$) is well defined. It is normal whenever $x$ is normal. Also, for normal $x$, we have
\begin{equation}\label{SkpTensor}
\skp{\sum_{n=0}^{+\infty}x^{\otimes n}}{\sum_{n=0}^{+\infty}x^{\otimes n}}=\sum_{n=0}^{+\infty}\skp{x^{\otimes n}}{x^{\otimes n}}=\sum_{n=0}^{+\infty}\skp xx^n=(1-\skp xx)^{-1}.
\end{equation}
\end{example}

We shall deal with \emph{unitarily invariant norms} on the algebra $B(H)$ of all bounded Hilbert space operators. For more details, the reader is referred to \cite[Chapter III]{GKrein}. We use the following facts. For any unitarily invariant norm $\luin\cdot\ruin$, we have $\luin A\ruin=\luin A^*\ruin=\luin\,|A|\,\ruin=\luin UAV\ruin=\luin A\ruin$ for all unitaries $U$ and $V$, as well as $\|A\|\le\luin A\ruin\le\|A\|_1$. The latter allows the following interpolation Lemma.

\begin{lemma}\label{Interpolacija} Let $T$ and $S$ be linear mappings defined on the space $\Sc_\infty$ of all compact operators on Hilbert space $H$. If
$$\|Tx\|\le\|Sx\|\mbox{ for all }x\in \Sc_\infty,\qquad\|Tx\|_1\le\|Sx\|_1\mbox{ for all }x\in
\Sc_1$$
then
$$\luin Tx\ruin\le\luin Sx\ruin$$
for all unitarily invariant norms.
\end{lemma}

\begin{proof} The norms $\|\cdot\|$ and $\|\cdot\|_1$ are dual to each other, in the sense
$$\|x\|=\sup_{\|y\|_1=1}|\tr(xy)|,\qquad \|x\|_1=\sup_{\|y\|=1}|\tr(xy)|.$$
Hence, $\|T^*x\|\le\|S^*x\|$, $\|T^*x\|_1\le\|S^*x\|_1$.

Consider the Ky Fan norm $\|\cdot\|_{(k)}$. Its dual norm is $\|\cdot\|_{(k)}^\sharp=\max\{\|\cdot\|,(1/k)\|\cdot\|_1\}$. Thus, by duality, $\|Tx\|_{(k)}\le\|Sx\|_{(k)}$ and the result follows by Ky Fan dominance property, \cite[\S3.4]{GKrein}.
\end{proof}

\section{Cauchy-Schwartz inequalities}

Cauchy-Schwartz inequality for $\|\cdot\|$ follows from (\ref{CS}), for $\|\cdot\|_1$ by duality and for other norms by interpolation.

\begin{theorem}\label{Prva} Let $A$ be a semifinite von Neumann algebra, let $M$ be a conjugated $W^*$-module over $A$ and let $a\in A$. Then:
\begin{equation}\label{1infty}
\|\skp x{ay}\|\le\|x\|\|y\|\|a\|,\qquad\|\skp x{ay}\|_1\le\|\skp{\ovl x}{\ovl x}^{1/2}a\skp{\ovl y}{\ovl y}^{1/2}\|_1;
\end{equation}
\begin{equation}\label{C2}
\|\skp x{ay}\|_2\le\|x\|\|a\skp{\ovl{y}}{\ovl{y}}^{1/2}\|_2,\quad\mbox{and}\quad
\|\skp x{ay}\|_2\le\|y\|\|\skp{\ovl{x}}{\ovl{x}}^{1/2}a\|_2.
\end{equation}

Specially, if $A=B(H)$, $\tau=\tr$ and $x$, $y$ normal, then
\begin{equation}\label{UIN1}
\luin\skp x{ay}\ruin\le\luin\skp xx^{1/2}a\skp yy^{1/2}\ruin
\end{equation}
for all unitarily invariant norms $\luin\cdot\ruin$.
\end{theorem}

\begin{proof} By (\ref{CS}), we have $\|\skp x{ay}\|\le\|x\|\|ay\|\le\|x\|\|y\|\|a\|$, which proves the first inequality in (\ref{1infty}).

For the proof of the second, note that by (\ref{Conjug}), for all $a\in L^1(A;\tau)$ we have $\tau(b\skp x{ay})=\tau(\skp{xb^*}{ay})=\tau(\skp{\ovl ya^*}{b\ovl x})=\tau(a\skp{\ovl y}{b\ovl x})$. Hence for $\skp{\ovl x}{\ovl x}$, $\skp{\ovl y}{\ovl y}\le1$
$$\|\skp x{ay}\|_1=\sup_{\|b\|=1}|\tau(b\skp x{ay})|=\sup_{\|b\|=1}|\tau(a\skp{\ovl y}{b\ovl x})|\le\|a\|_1\|\skp{\ovl y}{b\ovl x}\|\le\|a\|_1,$$

In the general case, let $\e>0$ be arbitrary, and let $x_1=(\skp{\ovl x}{\ovl x}+\e)^{-1/2}x$ and $y_1=(\skp{\ovl y}{\ovl y}+\e)^{-1/2}y$. Then $\ovl{x_1}=\ovl x(\skp{\ovl x}{\ovl x}+\e)^{-1/2}$ and $\ovl{y_1}=\ovl y(\skp{\ovl y}{\ovl y}+\e)^{-1/2}$ (by (\ref{Conjug})). Thus
$$\skp{\ovl{x_1}}{\ovl{x_1}}=(\skp xx+\e)^{-1/2}\skp xx(\skp xx+\e)^{-1/2}\le1,$$
by continuous functional calculus. Hence
\begin{align}\label{Racun}
\|\skp x{ay}\|_1&=\left\|\skp{(\skp{\ovl{x}}{\ovl{x}}+\e)^{1/2}x_1}{a(\skp{\ovl{y}}{\ovl{y}}+\e)^{1/2}y_1}\right\|_1=\nonumber\\
    &=\left\|\skp{x_1}{(\skp{\ovl{x}}{\ovl{x}}+\e)^{1/2}a(\skp{\ovl{y}}{\ovl{y}}+\e)^{1/2}y_1}\right\|_1\le\\
    &=\left\|(\skp{\ovl{x}}{\ovl{x}}+\e)^{1/2}a(\skp{\ovl{y}}{\ovl{y}}+\e)^{1/2}\right\|_1,\nonumber
\end{align}
and let $\e\to0$. (Note $\|(\skp{\ovl x}{\ovl x}+\e)^{1/2}-\skp{\ovl x}{\ovl x}^{1/2}\|\le\e^{1/2}$.)

To prove (\ref{C2}), by (\ref{CS}) we have
\begin{equation}\label{Refinement}
|\skp x{ay}|^2\le\|x\|^2\skp{ay}{ay}=\|x\|^2\skp y{a^*ay}.
\end{equation}
Apply $\|\cdot\|_1$ to the previous inequality. By (\ref{1infty}) we obtain
$$\|\skp x{ay}\|_2^2\le\|x\|^2\|\skp y{a^*ay}\|_1\le\|x\|^2\|\skp{\ovl{y}}{\ovl{y}}^{\frac12}a^*a\skp{\ovl{y}}{\ovl{y}}^{\frac12}\|_1=\|x\|^2\|a\skp{\ovl{y}}{\ovl{y}}^{\frac12}\|_2^2.$$
This proves the first inequality in (\ref{C2}). The second follows from duality
$$\|\skp x{ay}\|_2=\|\skp y{a^*x}\|_2\le\|y\|\|a^*\skp{\ovl{x}}{\ovl{x}}^{1/2}\|_2=\|y\|\|\skp{\ovl{x}}{\ovl{x}}^{1/2}a\|_2.$$

Finally, if $A=B(H)$, $\tau=\tr$ and $x$, $y$ normal. Then (\ref{UIN1}) holds for $\|\cdot\|_1$ by (\ref{1infty}). For the operator norm, it follows by normality. Namely then $x\skp xx=\skp xxx$ and we can
repeat argument from (\ref{Racun}). Now, the general result follows from Lemma \ref{Interpolacija}
\end{proof}

\begin{corollary} If $A=B(H)$ and $M=l^2(A)'$ (Example \ref{Standard}), then (\ref{UIN1}) is \cite[Theorem 2.2]{BananaPAMS1998} (the first formula from the abstract). If $M=L^2(\Omega,A)$, $A=B(H)$, (\ref{UIN1}) is \cite[Theorem 3.2]{BananaJFA2005} (the second formula from the abstract).
\end{corollary}

\begin{remark} The inequality (\ref{Refinement}) for $M=B(H)^n$ is proved in \cite{BananaStefan} using complicated identities and it plays an important role in this paper.
\end{remark}

Using three line theorem (which is a standard procedure), we can interpolate results of Theorem \ref{Prva} to $L^p(A,\tau)$ spaces.

\begin{theorem} Let $A$ be a semifinite von Neumann algebra, and let $M$ be a conjugated $W^*$-module over $A$. For all $p$, $q$, $r>1$ such that $1/q+1/r=2/p$, we have
\begin{equation}\label{InterP}
\|\skp x{ay}\|_p\le\left\|\skp{\skp xx^{q-1}\ovl x}{\ovl x}^{1/2q}a\skp{\skp yy^{r-1}\ovl y}{\ovl y}^{1/2r}\right\|_p.
\end{equation}
\end{theorem}

\begin{proof} Let $u$, $v\in M_0$ and let $b\in A$. For $0\le\re\lambda,\re\mu\le1$ consider the function
$$f(\lambda,\mu)=\skp{(\skp{\ovl u}{\ovl u}+\e)^{-\frac\lambda2}u(\skp uu+\e)^{\frac{\lambda-1}2}}{b(\skp{\ovl v}{\ovl v}+\e)^{-\frac\mu2}v(\skp vv+\e)^{\frac{\mu-1}2}}.$$
This is an analytic function (obviously).

On the boundaries of the strips, we estimate. For $\re\lambda=\re\mu=0$
$$f(it,is)=\skp{(\skp{\ovl u}{\ovl u}+\e)^{-\frac{it}2}u(\skp uu+\e)^{-\frac12+\frac{it}2}}{b(\skp{\ovl v}{\ovl v}+\e)^{-\frac{is}2}v(\skp vv+\e)^{-\frac12+\frac{is}2}}.$$
Since $(\skp{\ovl u}{\ovl u}+\e)^{-it/2}$, $(\skp uu+\e)^{it/2}$, $(\skp{\ovl v}{\ovl v}+\e)^{-is/2}$ and $(\skp vv+\e)^{is/2}$ are unitary operators, and since the norm of $u(\skp uu+\e)^{-1/2}$, $v(\skp vv+\e)^{-1/2}$ does not exceed $1$, by (\ref{1infty}) we have
\begin{equation}\label{00}
\|f(it,is)\|\le\|b\|.
\end{equation}

For $\re\lambda=\re\mu=1$
$$f(1+it,1+is)=\skp{(\skp{\ovl u}{\ovl u}+\e)^{-\frac12-\frac{it}2}u(\skp uu+\e)^{\frac{it}2}}{b(\skp{\ovl v}{\ovl v}+\e)^{-\frac12-\frac{is}2}v(\skp vv+\e)^{\frac{is}2}}.$$
By a similar argument, by (\ref{1infty}) we obtain
\begin{equation}\label{11}
\|f(1+it,1+is)\|_1\le\|b\|_1.
\end{equation}

For $\re\lambda=0$, $\re\mu=1$, by (\ref{C2}) we have
$$f(it,1+is)=\skp{(\skp{\ovl u}{\ovl u}+\e)^{-\frac{it}2}u(\skp uu+\e)^{-\frac12+\frac{it}2}}{b(\skp{\ovl v}{\ovl v}+\e)^{-\frac12-\frac{is}2}v(\skp vv+\e)^{\frac{is}2}},$$
and hence
\begin{equation}\label{01}
\|f(it,1+is)\|_2\le\|b\|_2.
\end{equation}
Similarly
\begin{equation}\label{10}
\|f(1+it,is)\|_2\le\|b\|_2.
\end{equation}

Let us interpolate between (\ref{00}) and (\ref{01}). Let $r>1$. Then $1/r=\theta\cdot1+(1-\theta)\cdot0$ for $\theta=1/r$. Then $(\theta\cdot(1/2)+(1-\theta)\cdot0)^{-1}=2r$ and hence, by three line theorem (see \cite{KosakiJFA1984} and \cite{Kosaki}) we obtain
\begin{equation}\label{0r}
\|f(0+it,1/r)\|_{2r}\le\|b\|_{2r}
\end{equation}

Similarly, interpolating between (\ref{10}) and (\ref{11}) we get
\begin{equation}\label{1r}
\|f(1+it,1/r)\|_{\frac{2r}{r+1}}\le\|b\|_{\frac{2r}{r+1}}
\end{equation}
since $1/r=\theta\cdot1+(1-\theta)\cdot0$ for same $\theta=1/r$ and $(\theta\cdot1+(1-\theta)\cdot(1/2))^{-1}=2r/(r+1)$

Finally, interpolate between (\ref{0r}) and (\ref{1r}). Then $1/q=\theta\cdot1+(1-\theta)\cdot0$ for $\theta=1/q$ and therefore $(\theta\cdot\frac{r+1}{2r}+(1-\theta)\frac1{2r})^{-1}=\frac1{2qr}\cdot(r+1+q-1)=p$. Thus
$$\|f(1/q,1/r)\|_p\le\|b\|_p.$$
i.e.
$$\left\|\skp{(\skp{\ovl u}{\ovl u}+\e)^{-\frac1{2q}}u(\skp uu+\e)^{\frac{1-q}{2q}}}{b(\skp{\ovl v}{\ovl v}+\e)^{-\frac1{2r}}v(\skp vv+\e)^{\frac{1-q}{2q}}}\right\|_p\le\|b\|_p.$$
After substitutions
$$u=x\skp xx^{(q-1)/2},\quad v=\skp yy^{(r-1)/2},\quad b=(\skp{\ovl u}{\ovl u}+\e)^{1/2q}a(\skp{\ovl v}{\ovl v}+\e)^{1/2r},$$
we obtain
\begin{multline*}\left\|\skp{x\skp xx^{(q-1)/2}(\skp xx^q+\e)^{(1-q)/2q}}{ay\skp yy^{(q-1)/2}(\skp yy^q+\e)^{(1-q)/2q}}\right\|_p\le\\\le\left\|\left(\skp{\skp xx^{q-1}\ovl x}{\ovl x}+\e\right)^{1/2q}a\left(\skp{\skp yy^{r-1}\ovl y}{\ovl y}+\e\right)^{1/2r}\right\|_p
\end{multline*}
which after $\e\to0$ yields (\ref{InterP}), using the argument similar to that in \cite[Lemma 1.3.9]{Manuilov}.
\end{proof}

\begin{remark} In a special case $A=B(H)$, $\tau=\tr$, $M=l^2(A)'$, $r=q=p$ formula (\ref{InterP}) becomes \cite[Theorem 2.1]{BananaJLMS1999} (the main result).

Also, for $A=B(H)$, $\tau=\tr$, $M=L^2(\Omega,A)$ formula (\ref{InterP}) becomes \cite[Theorem 3.3]{BananaJFA2005} (the first displayed formula from the abstract), there proved with an additional assumption that $\Omega$ is $\sigma$-finite.
\end{remark}

In the next two section we derive some inequalities that regularly arise from Cauchy-Schwartz inequality.

\section{Inequalities of the type $|1-\skp xy|\ge(1-\|x\|^2)^{1/2}(1-\|y\|^2)^{1/2}$}

The basic inequality can be proved as
\begin{align*}
|(1-\skp xy)^{-1}|&\le\sum_{n=0}^{+\infty}|\skp xy|^n\le\sum_{n=0}^{+\infty}\|x\|^n\|y\|^n\le\\
    &\left(\sum_{n=0}^{+\infty}\|x\|^{2n}\right)^{1/2}\left(\sum_{n=0}^{+\infty}\|y\|^{2n}\right)^{1/2}=
    (1-\|x\|^2)^{-1/2}(1-\|y\|^2)^{-1/2}.
\end{align*}

Following this method we prove:

\begin{theorem}\label{Naopaka} Let $M$ be a conjugated $W^*$-module over $A=B(H)$, let $x$, $y\in M_0$ be normal, and let $\skp xx$, $\skp yy\le1$. Then
$$\luin(1-\skp xx)^{1/2}a(1-\skp yy)^{1/2}\ruin\le\luin a-\skp x{ay}\ruin$$
in any unitarily invariant norm.
\end{theorem}

\begin{proof} We use examples \ref{Interior} and \ref{Fock}.

Denote $Ta=\skp x{ay}$. We have $T^2a=\skp x{\skp x{ay}y}=\skp{x\otimes x}{a y\otimes y}$ and by induction $T^ka=\skp{x^{\otimes k}}{a y^{\otimes k}}$. Suppose $\|x\|$, $\|y\|\le\delta<1$. Then $\|x^{\otimes k}\|$, $\|y^{\otimes k}\|\le\delta^k$ and hence $\luin T^k\ruin\le\delta^{2k}$. Then
\begin{equation}\label{Series}
(I-T)^{-1}=\sum_{n=0}^{+\infty}T^k.
\end{equation}
Put $b=(I-T)^{-1}a$. Then
\begin{align}\label{Tenzor}
\luin b\ruin&=\luin\sum_{k=0}^{+\infty}T^ka\ruin=\luin\sum_{k=0}^{+\infty}\skp{x^{\otimes k}}{ay^{\otimes k}}\ruin=\luin\skp{\sum_{k=0}^{+\infty}x^{\otimes k}}{a\sum_{k=0}^{+\infty}y^{\otimes k}}\ruin\le\\
    &\le\luin\skp{\sum_{k=0}^{+\infty}x^{\otimes k}}{\sum_{k=0}^{+\infty}x^{\otimes k}}^{1/2}a
    \skp{\sum_{k=0}^{+\infty}y^{\otimes k}}{\sum_{k=0}^{+\infty}y^{\otimes k}}^{1/2}\ruin.\nonumber
\end{align}
by (\ref{UIN1}) and normality of $x$ and $y$. Invoking (\ref{SkpTensor}), inequality (\ref{Tenzor}) becomes
\begin{equation}\label{I-T}
\luin(I-T)^{-1}a\ruin\le\luin(1-\skp xx)^{-1/2}a(1-\skp yy)^{-1/2}\ruin.
\end{equation}
Finally, note that the mappings $I-T$ and $a\mapsto(1-\skp xx)^{-1/2}a(1-\skp yy)^{-1/2}$ commute (by normality of $x$ and $y$) and put $(1-\skp xx)^{-1/2}(a-Ta)(1-\skp yy)^{-1/2}$ in place of $a$, to obtain the conclusion.

If $\|x\|$, $\|y\|=1$ then put $\delta x$ instead of $x$ and let $\delta\to1-$.
\end{proof}

\begin{remark} If $M=L^2(\Omega,A)$ this is \cite[Theorem 4.1]{BananaJFA2005} (the last formula from the abstract). If $M=B(H)\times B(H)$, $x=(I,A)$, $y=(I,B)$ then it is \cite[Theorem 2.3]{BananaPAMS1998} (the last formula from the abstract).
\end{remark}

\begin{remark} Instead of $t\mapsto 1-t$ we may consider any other function $f$ such that $1/f$ is well defined on some $[0,c)$ and has Taylor expansion with positive coefficients, say $c_n$. Then distribute $\sqrt{c_n}$ on both arguments in inner product in (\ref{Tenzor}) and after few steps we get
$$\luin (f(x^*x))^{1/2}a(f(y^*y))^{1/2}\ruin\le\luin f(T)\ruin.$$

For instance, for $t\mapsto(1-t)^\a$, $\a>0$ we have $(1-t)^{-\a}=\sum c_nt^n$, where $c_n=\Gamma(n+\a)/(\Gamma(\a)n!)>0$ and we get
\begin{equation}\label{AOTalpha}
\luin(1-\skp xx)^{\a/2}a(1-\skp yy)^{\a/2}\ruin\le\luin(I-T)^\a a\ruin
\end{equation}
in any unitarily invariant norm. For $M=A=B(H)$, (\ref{AOTalpha}) reduces to
$$\luin(1-x^*x)^{\a/2}a(1-y^*y)^{\a/2}\ruin\le\luin\sum_{n=0}^{+\infty}(-1)^n\binom anx^{*n}ay^n\ruin,$$
which is the main result of \cite{StefanAOT2016}. Varying $f$, we may obtain many similar inequalities.
\end{remark}

Finally, if normality condition on $x$ and $y$ is dropped, we can use (\ref{InterP}) to obtain some inequalities in $L^p(A;\tau)$ spaces.

\begin{theorem}Let $M$ be a conjugated $W^*$-module over a semifinite von Neumann algebra $A$, let $x$, $y\in M_0$, $\|x\|$, $\|y\|<1$ and let
\begin{equation}\label{Defekt}
\Delta_z=\skp{\sum_{n=0}^{+\infty}z^{\otimes n}}{\sum_{n=0}^{+\infty}z^{\otimes n}}^{-1/2},\quad\mbox{for}~z\in\{x,y,\ovl x,\ovl y\}.
\end{equation}
Then
$$\|\Delta_x^{1-1/q}a\Delta_y^{1-1/r}\|_p\le\|\Delta_{\ovl x}^{-1/q}(a-\skp x{ay})\Delta_{\ovl y}^{-1/r}\|_p,$$
for all $p$, $q$, $r>1$ such that $1/q+1/r=2/p$.
\end{theorem}

\begin{proof} Let $b=(I-T)a$. We have $a=(I-T)^{-1}b$ and hence
\begin{align*}\|\Delta_x^{1-1/q}a\Delta_y^{1-1/r}\|_p&=\Big\|\Delta_x^{1-1/q}\sum_{n=0}^{+\infty}\skp{x^{\otimes n}}{by^{\otimes n}}\Delta_y^{1-1/r}\Big\|_p=\\
&=\Big\|\sum_{n=0}^{+\infty}\skp{x^{\otimes n}\Delta_x^{1-1/q}}{by^{\otimes n}\Delta_y^{1-1/r}}\Big\|_p\le\|ubv\|_p,
\end{align*}
by (\ref{InterP}), where
$$u=\sum_{n=0}^{+\infty}\skp{\sum_{n=0}^{+\infty}\skp{x^{\otimes n}\Delta_x^{1-1/q}}{x^{\otimes n}\Delta_x^{1-1/q}}^{q-1}\Delta_x^{1-1/q}\ovl x^{\otimes n}}{\Delta_x^{1-1/q}\ovl x^{\otimes n}}^{1/2q}.$$
After a straightforward calculation, we obtain $u=\Delta_{\ovl x}^{-1/q}$ and similarly $v=\Delta_{\ovl y}^{-1/r}$ and the conclusion follows.
\end{proof}

\begin{remark} When $A=B(H)$, $\tau=\tr$, this is the main result of \cite{LudaciFilomat}, from which we adapted the proof for our purpose. However, unaware of Fock module technique, the authors of \cite{LudaciFilomat} produced significantly more robust formulae.

Also, in \cite{LudaciFilomat}, the assumptions are relaxed to $r(T_{x,x})$, $r(T_{y,y})\le1$, where $r$ stands for the spectral radius and $T_{x,y}(a)=\skp x{ay}$. This easily implies $r(T_{x,y})\le1$. First, it is easy to see that $\|T_{z,z}\|=\|z\|^2$. Indeed, by (\ref{1infty}) we have $\|T_{z,z}\|\le\|z\|^2$. On the other hand, choosing $a=1$ we obtain $\|T_{z,z}\|\ge\|T_{z,z}(1)\|=\|\skp zz\|=\|z\|^2$. Again, by (\ref{1infty}), we have $\|T_{x,y}\|\le\|x\|\|y\|=\sqrt{\|T_{x,x}\|\|T_{y,y}\|}$. Apply this to $x^{\otimes n}$ and $y^{\otimes n}$ instead of $x$ and $y$ and we get $\|T^n_{x,y}\|\le\sqrt{\|T_{x,x}^n\|\|T_{y,y}^n\|}$ from which we easily conclude $r(T_{x,y})^2\le r(T_{x,x})r(T_{y,y})$ by virtue of spectral radius formula. (In a similar way, we can conclude $r(T_{x,x})=\|x\|$ for normal $x$.)

Thus, if both $r(T_{x,x})$, $r(T_{y,y})<1$, the series in (\ref{Defekt}) converge. If some of $r(T_{x,x})$, $r(T_{y,y})=1$ then define $\Delta_x=\lim_{\delta\to0}\Delta_{\delta x}=\inf_{0<\delta<1}\Delta_{\delta x}$, etc, and the result follows, provided that series that defines $\Delta_{\ovl x}$ and $\Delta_{\ovl y}$ are weakly convergent.
\end{remark}

\section{Gr\"uss type inequalities}

For classical Gr\"uss inequality, see \cite[\S2.13]{Mitrin}. We give a generalization to Hilbert modules following very simple approach from \cite{Dragomir} in the case of Hilbert spaces.

\begin{theorem} Let $M$ be a conjugated $W^*$-module over $B(H)$, and let $e\in M$ be such that $\skp ee=1$. The mapping $\Phi:M\times M\to B(H)$, $\Phi(x,y)=\skp xy-\skp xe\skp ey$. is a semi-inner product.

If, moreover $x$, $y\in M_0$ are normal with respect to $\Phi$. Then
\begin{equation}\label{Gruss3}
\luin\skp x{ay}-\skp xe\skp e{ay}\ruin\le\luin(\skp xx-|\skp xe|^2)^{1/2}a(\skp yy-|\skp ye|^2)^{1/2}\ruin
\end{equation}
in any unitarily invariant norm.

Finally, if $x$, $y$ belongs to balls with diameters $[me,Me]$ and $[pe,Pe]$ ($m$, $M$, $p$, $P\in\R$), respectively, then
\begin{equation}\label{GrussMm}
\luin\skp x{ay}-\skp xe\skp e{ay}\ruin\le\frac14\luin a\ruin|M-m||P-p|.
\end{equation}

(Here, $x$ belongs to the ball with diameter $[y,z]$ iff $\|x-\frac{y+z}2\|\le\|\frac{z-y}2\|$.)
\end{theorem}

\begin{proof} The mapping $\Phi$ is obviously linear in $y$ and conjugate linear in $x$. Moreover, by inequality (\ref{CS})
$$\skp xe\skp ex=|\skp ex|^2\le\|e\|^2\skp xx=\skp xx,$$
i.e.\ $\Phi(x,x)\ge0$. Hence $\Phi$ is an $A$-valued (semi)inner product. 

If $x$, $y$ are normal, then, by (\ref{UIN1}) we obtain
\begin{equation}\label{GrussPhi}
\luin\Phi(x,ay)\ruin\le\luin\Phi(x,x)^{1/2}a\Phi(y,y)^{1/2}\ruin
\end{equation}
in any unitarily invariant norm. Write down exact form of $\Phi$ and we obtain (\ref{Gruss3}).

Finally, for the last conclusion, note that $\Phi(x,x)=\Phi(x-ec,x-ec)$ for any $c\in\C$ (direct verification), and hence $\Phi(x,x)\le\skp{x-ec}{x-ec}$, which implies $\|\Phi(x,x)^{1/2}\|\le\|x-ec\|$. Choosing $c=(M+m)/2$, we obtain $\|\Phi(x,x)^{1/2}\|\le(M-m)/2$. Similarly, $\|\Phi(y,y)^{1/2}\|\le(P-p)/2$. Thus (\ref{GrussPhi}) implies (\ref{GrussMm}).
\end{proof}

\begin{remark}Choose $M=L^2(\Omega,\mu)$, $\mu(\Omega)=1$ and choose $e$ to be the function identically equal to $1$. Then
$$\Phi(x,ay)=\int_\Omega x(t)^*ay(t)\d\mu(t)-\int_\Omega x(t)^*\d\mu(t)\int_\Omega ay(t)\d\mu(t),$$
and from (\ref{Gruss3}) and (\ref{GrussMm}) we obtain main results of \cite{BananaGeorg}.
\end{remark}

\begin{remark} Applying other inequalities from section 3, we can derive other results from \cite{BananaGeorg}. Also, applying inequality $|\skp x{ay}|^2\le\|x\|^2\skp{ay}{ay}$ to the mapping $\Phi$ instead of $\skp\cdot\cdot$ we obtain the key result of \cite{Ludaci5autora}, there proved by complicated identities.
\end{remark}

\section{Concluding remarks}

Both, elementary operators and i.p.t.i.\ transformers on $B(H)$ are special case of
\begin{equation}\label{ElOpSkp}
\skp x{Ty}
\end{equation}
where $x$, $y$ are vectors from some Hilbert $W^*$-module $M$ over $B(H)$ and $T:M\to M$ is given by left action of $B(H)$.

Although there are many results independent of the representation (\ref{ElOpSkp}), a lot of inequalities related to elementary operators and i.p.t.i.\ transformers can be reduced to elementary properties of the $B(H)$-valued inner product.

\bibliographystyle{abbrv}
\bibliography{CauchySch}

\end{document}